\documentclass[11pt]{article}
\usepackage{amsmath,amsthm,amssymb,stackengine}
\usepackage[hyperfootnotes=false]{hyperref}
\usepackage{color}
\usepackage{cancel}
\usepackage{titlesec}
\titleformat{\paragraph}
{\normalfont\normalsize\bfseries}{\theparagraph}{1em}{}
\titlespacing*{\paragraph}
{0pt}{3.25ex plus 1ex minus .2ex}{1.5ex plus .2ex}

\def\titlerunning#1{\gdef\titrun{#1}}
\makeatletter
\def\author#1{\gdef\autrun{\def\and{\unskip, }#1}\gdef\@author{#1}}
\def\address#1{{\def\and{\\\hspace*{18pt}}\renewcommand{\thefootnote}{}%
\footnote {#1}}%
\markboth{\autrun}{\titrun}}
\makeatother
\def\email#1{\hspace*{4pt}{\em e-mail}: #1}
\def\MSC#1{{\renewcommand{\thefootnote}{}%
\footnote{\emph{Mathematics Subject Classification (2020):} #1}}}
\def\keywords#1{\par\medskip
\noindent\textbf{Keywords:} #1}

\newcommand{\nsq}{\ensurestackMath{\stackinset{c}{}{c}{}{/}{\square}}}

\newtheorem{theorem}{Theorem}[section]

\newtheorem{prop}[theorem]{Proposition}

\newtheorem{lemma}[theorem]{Lemma}

\theoremstyle{definition}

\newtheorem{remark}[theorem]{Remark}

\numberwithin{equation}{section}

\def\cA{\mathcal A}
\def\cB{\mathcal B}
\def\cC{\mathcal C}
\def\cD{\mathcal D}
\def\cE{\mathcal E}

\def\cO{\mathcal O}

\def\cQ{\mathcal Q}

\def\cT{\mathcal T}

\def\cW{\mathcal W}
\def\cX{\mathcal X}

\def\PG{{\rm PG}}

\def\F{{\mathbb F}}

\def\PGL{{\rm PGL}}
\def\PGO{{\rm PGO}}
\def\PSL{{\rm PSL}}
\def\PSp{{\rm PSp}}

\def\Tr{{\rm Tr}}

\def\i{\boldsymbol i}

\begin{document}


\baselineskip=16pt

\titlerunning{}

\title{On near--MDS codes and caps}

\author{Michela Ceria
\and 
Antonio Cossidente
\and
Giuseppe Marino 
\and 
Francesco Pavese}

\date{}

\maketitle

\address{M. Ceria: Dipartimento di Meccanica, Matematica e Management, Politecnico di Bari, Via Orabona 4, 70125 Bari, Italy; \email{michela.ceria@poliba.it}
\and
A. Cossidente: Dipartimento di Matematica, Informatica ed Economia, Universit{\`a} degli Studi della Basilicata, Contrada Macchia Romana, 85100, Potenza, Italy; \email{antonio.cossidente@unibas.it}
\and
G. Marino: Dipartimento di Matematica e Applicazioni ``Renato Caccioppoli'', Universit{\`a} degli Studi di Napoli ``Federico II'', Complesso Universitario di Monte Sant'Angelo, Cupa Nuova Cintia 21, 80126, Napoli, Italy; \email{giuseppe.marino@unina.it}
\and
F. Pavese: Dipartimento di Meccanica, Matematica e Management, Politecnico di Bari, Via Orabona 4, 70125 Bari, Italy; \email{francesco.pavese@poliba.it}
}


\MSC{Primary 51E22; 94B05. Secondary 51E20.}

\begin{abstract}
Several classes of near--MDS codes of $\PG(3, q)$ are described. They are obtained either by considering the intersection of an elliptic quadric ovoid and a Suzuki--Tits ovoid of a symplectic polar space $\cW(3, q)$ or starting from the $q+1$ points of a twisted cubic of $\PG(3, q)$. As a by-product two classes of complete caps of $\PG(4, q)$ of size $2q^2-q\pm\sqrt{2q}+2$ are exhibited.

\keywords{Near--MDS code, cap, ovoid, twisted cubic.}
\end{abstract}

\section{Introduction}

A {\em $q$--ary linear code} $C$ of dimension $k$ and length $N$ is a $k$-dimensional vector subspace of $\F_q^N$, whose elements are called {\em codewords}. A {\em generator matrix} of $C$ is a matrix whose rows form a basis of $C$ as an $\F_q$-vector space. The {\em minimum distance of $C$} is $d = \min\{d(u, 0) \mid u \in C, u \ne 0\}$, where $d(u, v)$, $u, v \in \F_q^N$, is the {\em Hamming distance} on $\F_q^N$, namely the number of different components between $u$ and $v$. A vector $u$ is {\em $\rho$--covered by $v$} if $d(u, v) \le \rho$. The {\em covering radius} of a code $C$ is the smallest integer $\rho$ such that every vector of $\F_q^n$ if $\rho$--covered by at least one codeword of $C$. A linear code with minimum distance $d$ and covering radius $\rho$ is said to be an $[N, k, d]_q$ $\rho$--code. Sometimes $d$ and $\rho$ are omitted and the notation $[N, k]_q$ code is used. For a code $C$, its {\em dual code} is $C^\perp = \{v \in \F_q^N \mid v \cdot c = 0, \forall c \in C\}$ (here $\cdot$ is the Euclidean inner product). The dimension of the dual code $C^\perp$ or the codimension of $C$ is $N - k$. Any matrix which is a generator matrix of $C^\perp$ is called a {\em parity check matrix} of $C$. If $C$ is linear with parity check matrix $M$, its covering radius is the smallest $\rho$ such that every $w \in \F_q^{N - k}$ can be written as a linear combination of at most $\rho$ columns of $M$. 

Let $\PG(k-1, q)$ be the $(k-1)$-dimensional projective space over the finite field $\F_q$ and let $X_1, X_2, \ldots, X_{k}$ be homogeneous projective coordinates. We denote by $U_i$, $i=1,\ldots,k$ the point of $\PG(k-1,q)$ having $1$ in the $i$-th position and $0$ elsewhere. An {\em $n$--cap} of $\PG(k-1, q)$ is a set of $k$ points no three of which are collinear. An $n$--cap of $\PG(k-1, q)$ is said to be {\em complete} if it is not contained in an $(n+1)$--cap of $\PG(k-1, q)$. By identifying the representatives of the points of a complete $n$-cap of $\PG(k-1, q)$ with columns of a parity check matrix of a $q$--ary linear code it follows that (apart from three sporadic exceptions) complete $n$--caps in $\PG(k-1, q)$ with $n > k$ and non-extendable linear $[n, n - k, 4]_q$ $2$--codes are equivalent objects, see \cite{GDT}. One of the main issue is to determine the spectrum of the sizes of complete caps in a given projective space. The interested reader is referred to \cite{G1} and references therein for an account on the subject.

For an $[N, k, d]_q$ code the so called {\em Singleton bound} holds: $d \le N-k+1$; the integer $N - k + 1 - d$ is known as the {\em Singleton defect of $C$}. A code with zero Singleton defect is called {\em maximum distance separable} (or {\em MDS} for short), whereas a code $C$ such that both of $C$ and $C^\perp$ have Singleton defect one is said to be {\em near--MDS code}. In particular, an $[N, k]_q$ linear code $C$ is a near--MDS code if and only if the the columns of a generator matrix $G$ of $C$ satisfies the following conditions:
\begin{itemize}
    \item 
    any $k-1$ columns of $G$ are linearly independent,
    \item
    there exist $k$ linearly dependent columns in $G$,
    \item
    any $k+1$ columns of $G$ have full rank.
\end{itemize}
By considering the columns of $G$ as representatives of projective points of $\PG(k-1, q)$, $k \ge 3$, it follows that near--MDS codes are equivalent to subsets $\cX$ of $\PG(k-1,q)$ having the following properties:
\begin{itemize}
\item
every $k-1$ points of $\cX$ generate a hyperplane in $\PG(k-1,q)$,
\item
 there exist $k$ points in $\cX$ lying on a hyperplane,
 \item
 every $k+1$ points of $\cX$ generate $\PG(k-1,q)$.
\end{itemize}
The reader is referred to \cite{DL} for more details. Throughout the paper we will refer to a pointset of $\PG(k-1, q)$ satisfying the properties above as an {\em NMDS-set}. An NMDS-set is said to be {\em complete} if it is maximal with respect to set theoretical inclusion. The size of an NMDS-set of $\PG(k-1, q)$ is at most $2q+k-2$, if $q > 3$ and $2q+k$ otherwise \cite[Proposition 6.2, Proposition 5.1]{DL}. The largest known NMDS-sets arise from elliptic curves and have size $q + \lceil 2 \sqrt{q} \rceil$ if $ q = p^r$, $r \ge 3$ odd, and $p | \lceil 2 \sqrt{q} \rceil$ or $q + \lceil 2 \sqrt{q} \rceil + 1$ otherwise. The completeness of these codes has been investigated in \cite{Giulietti}. Further constructions of NMDS-sets have been provided in \cite{AL, AGS, WH}.

In this paper we deal with near--MDS codes of dimension $4$ and caps of $\PG(4, q)$. In Section~\ref{sec2} a class of NMDS-sets of $\PG(3, q)$, $q = 2^{2h+1}$, $h \ge 1$, having size $q^2+\sqrt{2q}+1$ is exhibited. It is obtained by looking at the intersection of an elliptic quadric ovoid and a Suzuki--Tits ovoid of a symplectic polar space $\cW(3, q)$. Basing on this result, in Section \ref{sec3} we describe two classes of complete caps of $\PG(4, q)$ of size $2q^2-q\pm\sqrt{2q}+2$. Finally in Section \ref{sub} we completely determine how many points an NMDS-set containing the $q+1$ points of a twisted cubic of $\PG(3, q)$ can have.

\section{NMDS-sets from ovoids of $\cW(3, q)$}\label{sec2}

In this section we study a class of NMDS-sets of $\PG(3, q)$, $q = 2^{2r+1}$, arising by intersecting an elliptic quadric and a Suzuki--Tits ovoid. Let $\cW(3, q)$ be a non-degenerate symplectic polar space of $\PG(3, q)$, i.e. the set of all totally isotropic points and totally isotropic lines (called {\em generators}) with respect to a (non-degenerate) alternating bilinear form of the vector space underlying $\PG(3, q)$. Thus, $\cW(3, q)$ consists of all the points of $\PG(3, q)$ and of $(q + 1)(q^2 + 1)$ generators. Through every point $P \in \PG(3, q)$ there pass $q + 1$ generators and these lines are coplanar. The plane containing these lines is the {\em polar plane} of $P$ with respect to the symplectic polarity defining $\cW(3, q)$. The incidence structure $\cW(3, q)$ is preserved by the projective symplectic group $\PSp(4, q)$. An {\em ovoid} $\cO$ of $\cW(3, q)$ is a set of $q^2+1$ points of $\cW(3, q)$ such that every generator of $\cW(3, q)$ meets $\cO$ in exactly one point. It is well known that $\cW(3, q)$ possesses no ovoid if $q$ is odd, whereas there are two known classes of ovoids if $q$ is even, namely the elliptic ovoid preserved by the group $\PSL(2, q^2)\cong {\rm P}\Omega^-(4, q)$ and the Suzuki--Tits ovoid admitting the group $Sz(2^{2r+1})$, \cite[Section 7.2]{HT}. A plane of $\PG(3, q)$ meets an ovoid of $\cW(3, q)$ in one point or $q+1$ points. The latter pointset is a conic if the ovoid is elliptic or a translation oval if the ovoid is of Suzuki--Tits type. Two elliptic ovoids of $\cW(3, q)$ meet in a point or in a conic. From \cite{BS}, two Suzuki--Tits ovoids of $\cW(3, q)$ have $1$, $q+1$, $2q+1$ or $q\pm\sqrt{2q}+1$ points in common, whereas an elliptic ovoid and a Suzuki--Tits ovoid of $\cW(3, q)$ meet in $q\pm\sqrt{2q}+1$ points. Let $\cW(3, q)$, $q=2^{2r+1}$ and $r\geq 1$, be given by the alternating bilinear form
\[
x_1y_4 + x_4y_1 + x_2y_3 + x_3y_2
\]
and denote by $\perp$ the associated symplectic polarity of $\PG(3, q)$. Let $\cT$ be a Suzuki--Tits ovoid of $\cW(3, q)$. We may assume w.l.o.g. that 
\[
\cT = \left\{(1,x_2,x_3,x_4) \mid x_2, x_3, x_4 \in \F_q, x_2 x_3 + x_2^{\sigma + 2} + x_3^{\sigma} + x_4 = 0\right\} \cup \{U_4\},  
\]
where $x^{\sigma} = x^{2^{r+1}}$ and hence $x^{\sigma^2} = x^2$. The Suzuki group $Sz(q) \le \PSp(4, q)$ leaving $\cT$ invariant has a $2$--transitive action on points of $\cT$ and has two orbits on points of $\cW(3, q)$ \cite[Section 16.4]{H2}. 

\begin{lemma}\label{lemma:suz}
Let $\cE$ be an elliptic ovoid of $\cW(3, q)$ and let $\cT$ be a Suzuki--Tits ovoid of $\cW(3, q)$. A plane of $\PG(3, q)$ has at most four points in common with $\cE \cap \cO$. 
\end{lemma}
\begin{proof}
Let $\pi$ be a plane of $\PG(3, q)$. We only need to consider the case when $|\pi \cap \cE| = |\pi \cap \cO| = q+1$. Since $Sz(q)$ is transitive on these planes we may assume $\pi: X_2 = 0$. Hence $\pi^\perp = U_3$ is the nucleus of the oval $\pi \cap \cT$. Observe that the conic $\pi \cap \cE$ has the same nucleus $\pi^\perp = U_3$, due to the fact that $\cE$ is an ovoid of $\cW(3, q)$. Hence, by \cite[Corollary 7.12]{H1}, $\pi \cap \cE$ is the set of points satisfying the quadratic equation
\[
a_{11} X_1^2 + a_{33} X_3^2 + a_{44} X_4^2 + X_1 X_4 = 0,
\] 
for some $a_{11}, a_{33}, a_{44} \in \F_q$, with $a_{33} \ne 0$. 
Suppose that $U_4$ is not a point of $\pi\cap\cE$. Then $a_{44}\ne 0$ and the point $P = (1,0,x,x^{\sigma}) \in \pi \cap \cT$ belongs to $\pi \cap \cE$ if and only if there exists $x \in \F_q$ such that \begin{align}
a_{11} + a_{33} x^2 + a_{44} x^{2\sigma} + x^{\sigma} = 0, \label{suz}
\end{align}
that is
\begin{align*}
0 = \left( a_{11} + a_{33} x^2 + a_{44} x^{2\sigma} + x^{\sigma} \right)^{2^r} = a_{11}^{2^r} + a_{33}^{2^r} x^\sigma + a_{44}^{2^r} x^{2} + x. 
\end{align*}
Therefore 
\[
x^\sigma = \left( \frac{a_{44}}{a_{33}} \right)^{2^r} x^2 + \frac{1}{a_{33}^{2^r}} x + \left( \frac{a_{11}}{a_{33}} \right)^{2^r}
\]
and substituting in Equation \eqref{suz}, we get that this equation has at most four solutions. If $P = U_4$, then $a_{44}=0$ and arguing as above we get that Equation \eqref{suz} has at most two solutions, i.e. the plane $\pi$ contains at most three point of $\cE\cap\cT$.
\end{proof}

\begin{prop}
Let $\cE$ be an elliptic ovoid of $\cW(3, q)$ and let $\cT$ be a Suzuki--Tits ovoid of $\cW(3, q)$ such that $|\cE \cap \cT| = q + \sqrt{2q} +1$. Thus $|\cE \cap \cT|$ is an NMDS-set.
\end{prop}
\begin{proof}
It is sufficient to observe that if $|\cE \cap \cT| = q + \sqrt{2q} + 1$, then there are planes intersecting $\cE \cap \cT$ in four points. Indeed, if on the contrary every plane intersects $\cE\cap\cT$ in at most three points, considering the plane through a secant line $\cE\cap\cT$ the number of points of $\cE\cap\cT$ covered by them is at most $q+3$, a contradiction.
\end{proof}

\begin{remark}
Some computations performed with the aid of Magma \cite{magma} show that the NMDS-set constructed in the previous proposition can be extended by adding two further points if $q = 8$ and it is complete if $q = 32$.
\end{remark}

\section{Complete caps of $\PG(4, q)$} \label{sec3}

In this section, on the basis of the results obtained in Section \ref{sec2}, we exhibit two classes of complete caps of $\PG(4, q)$, $q$ even, starting from two ovoids of a parabolic quadric. Let $\cQ(4, q)$ be the parabolic quadric of $\PG(4, q)$ defined by $X_1 X_5 + X_2 X_4 + X_3^2 = 0$. The quadric $\cQ(4, q)$ has $(q+1)(q^2+1)$ points and $(q + 1)(q^2 + 1)$ lines (or {\em generators}). Through every point $P \in \cQ(4, q)$ there pass $q + 1$ generators that are the lines of a quadratic cone. We will denote by $t_{P}$ the three-dimensional projective space containing this cone and we will refer to it as the {\em tangent space} to $\cQ(4, q)$ at $P$. The quadric $\cQ(4, q)$ has the point $U_3$ as a nucleus. If $R$ is a point of $\PG(4, q)$ not on $\cQ(4, q)$, let $P$ be the unique point in common between $\cQ(4, q)$ and the line $U_3 R$; thus the $q^2+q+1$ lines joining $R$ with the points of $t_P \cap \cQ$ are the lines that are tangent to $\cQ(4, q)$ and pass through $R$. Let $\PGO(5, q)$ denote the group consisting of the projectivities of $\PG(4, q)$ leaving invariant $\cQ(4, q)$. An {\em ovoid} $\cO$ of $\cQ(4, q)$ is a set of $q^2+1$ points of $\cQ(4, q)$ such that every generator of $\cQ(4, q)$ meets $\cO$ in exactly one point. Since $q$ is even, ovoids of $\cQ(4, q)$ and ovoids of $\cW(3, q)$ are equivalent objects. Indeed, by projecting the points of $\cQ(4, q)$ from $U_3$ onto a hyperplane $\Pi$ of $\PG(4, q)$ not containing $U_3$, the points and the lines of $\cQ(4, q)$ are mapped to the points and the lines of a symplectic polar space $\cW(3, q)$ of $\Pi$. Also, by projecting the conics of $\cQ(4, q)$ having $U_3$ as a nucleus one gets the lines of $\Pi$ that are not lines of $\cW(3, q)$. It turns out that elliptic ovoids of $\cW(3, q)$ correspond to three-dimensional hyperplane sections meeting $\cQ(4, q)$ in an elliptic quadric, whereas Suzuki--Tits ovoids of $\cW(3, q)$ correspond to ovoids of $\cQ(4, q)$ spanning the whole $\PG(4, q)$. If $\cT'$ is a Suzuki--Tits ovoid of $\cQ(4, q)$, $q = 2^{2r+1}$, we may assume w.l.o.g. that 
\[
\cT' = \left\{(1,x_2,x_3,x_4,x_5) \mid x_2, x_3, x_4, x_5 \in \F_q, x_4 + x_2^{\sigma+1} + x_3^{\sigma} = x_5 + x_2 x_3^{\sigma} + x_2^{\sigma + 2} + x_3^{2} = 0\right\} \cup \{U_5\},  
\]
where $x^{\sigma} = x^{2^{r+1}}$ and hence $x^{\sigma^2} = x^2$, see \cite{PW}. From the discussion above it follows that the Suzuki group $Sz(q) \le \PGO(5, q)$ leaving invariant $\cT'$ has a $2$--transitive action on points of $\cT'$ and has two orbits on points of $\cQ(4, q)$. 
\begin{lemma}
The group $Sz(q)$ has three orbits $\{U_3\}$, $\cO_1$, $\cO_2$ on points of $\PG(4, q) \setminus \cQ(4, q)$ of size $1$, $(q^2+1)(q-1)$ and $q(q-1)(q^2+1)$, respectively.
\end{lemma}
\begin{proof}
The group $Sz(q)$ has to fix the nucleus $U_3$. To see that it has two further orbits on points of $\PG(4, q) \setminus \cQ(4, q)$, it is sufficient to note that the subgroup of $Sz(q)$ of order $q-1$ given by 
\[
\begin{pmatrix}
1 & 0 & 0 & 0 & 0 \\
0 & d & 0 & 0 & 0 \\
0 & 0 & d^{\frac{\sigma+2}{2}} & 0 & 0 \\
0 & 0 & 0 & d^{\sigma+1} & 0 \\
0 & 0 & 0 & 0 & d^{\sigma+2}
\end{pmatrix}, 
d \in \F_q \setminus \{0\},
\]
permutes in a single orbit the $q-1$ points of both $U_1 U_{3} \setminus \{U_1, U_3\}$ and $U_2 U_3 \setminus \{U_2, U_3\}$.
\end{proof}

\begin{lemma}\label{O1vsO2}
A point of $\PG(4, q) \setminus \left(\cQ(4, q) \cup \{U_3\} \right)$ lies on $0$ or $q/2$ lines that are secant to $\cT'$ according as it belongs to $\cO_1$ or $\cO_2$, respectively.
\end{lemma}
\begin{proof}
A line $\ell$ that is secant to $\cT'$ has $q-1$ points belonging to $\cO_2$. Indeed, the plane $\langle \ell, U_3 \rangle$ meets $\cQ(4, q)$ in a conic that has $U_3$ as a nucleus. By projecting this conic from the nuclues $U_3$ onto a hyperplane $\Pi$ not passing through $U_3$, we get a line $m$ not of $\cW(3,q)$ having two points in common with the Suzuki ovoid $\cT$ of $\Pi$. If there was another point on the conic belonging to $\cT'$ then the line $m$ would intersect $\cT$ in at least three points, a contradiction since $\cT$ is an ovoid of $\Pi$. This means that each line joining $U_3$ with one of the $q-1$ points of $\ell\setminus \cT'$ intersects the quadric $\cQ(4,q)$ at a point not in $\cT'$. The statement follows by considering the incidence structure having as pointset $\cO_2$ and as blocks the $q^2(q^2+1)/2$ lines that are secant to $\cT'$, where incidence is the natural one. 
\end{proof}

\begin{prop}
Let $\cO, \cO'$ be two ovoids of $\cQ(4, q)$. Thus $\cO \cup \cO' \cup \{U_3\}$ is a cap of $\PG(4, q)$.
\end{prop}
\begin{proof}
Let $\cB = \cO \cup \cO' \cup \{U_3\}$. Every line of $\cQ(4, q)$ has $0, 1$ or $2$ points in common with $\cB$. Since every line through $U_3$ has exactly one point in common with $\cQ(4, q)$, it follows that $\cB$ is a cap. 
\end{proof}

\begin{theorem}
Let $\cE'$ be an elliptic quadric of $\cQ(4, q)$, $q = 2^{2r+1}$, and let $\cT'$ be a Suzuki--Tits ovoid of $\cQ(4, q)$. Thus $\cE' \cup \cT' \cup \{U_3\}$ is a complete cap of $\PG(4, q)$ of size $2q^2-q\pm\sqrt{2q}+2$.
\end{theorem}
\begin{proof}
Let $\cB = \cE' \cup \cT' \cup \{U_3\}$ and denote by $\Pi$ the three-dimensional projective space containing $\cE'$. By the previous proposition $\cB$ is a cap. 

We show that $\cB$ is complete. Let $P$ be a point of $\PG(4, q) \setminus (\Pi \cup \cQ(4, q) \cup \{U_3\})$. If $P$ belongs to $\cO_2$, by Lemma \ref{O1vsO2}, through $P$ there pass $q/2$ lines that are secant to $\cT' \subset \cB$, whereas if $P \in \cO_1$ then $U_3 P$ meets $\cT'$ in a point $Q$ belonging to $\cT'$ and hence $P$ lies on the line $U_3 Q$ that is secant to $\cB$. Let $P \in \cQ(4, q) \setminus \cB$. We claim that there is a generator through $P$ intersecting $\cB$ in two points. To this end it is enough to show that not all the $q+1$ lines of $\cQ(4, q)$ passing through $P$ meet $\cE' \cap \cT'$. Indeed, let $\cW(3, q)$ be the symplectic polar space of $\Pi$ obtained by projecting $\cQ(4, q)$ from $U_3$ onto $\Pi$. Thus $\cT'$ corresponds to $\cT$ and the $q+1$ generators of $\cQ(4, q)$ containing $P$ are mapped to the $q+1$ lines of $\cW(3, q)$ through the point $P U_3 \cap \Pi$. Since these latter lines lie in a plane, by Lemma \ref{lemma:suz} we have that at most four of them have at least one point in common with $\cE \cap \cT$. 
\end{proof}

\section{NMDS-sets containing a twisted cubic of $\PG(3, q)$}\label{sub}

\subsection{Some geometry of plane cubic curves}
The following preliminary result is based on \cite{BMP}. Here and in the sequel we will denote by $\cA_{\cD}$ the set of points of a plane $\pi$ not lying on none of the lines of $\pi$ sharing three points with a cubic curve $\cD$ of $\pi$. Let $\square_q$ and $\nsq_q$
denote the sets of non-zero squares and non-squares of $\F_q$, respectively. Also, we will denote by $\Tr(\cdot)$ the absolute trace from $\F_q$, $q$ even, to $\F_{2}$.

\begin{lemma}\label{plane}
Let $q\geq 23$ and let
\begin{align*}
	& \cD_1:  X_1 X_3^2 - X_2^3 = 0, \mbox{ then}\\
	& \cA_{\cD_1} = 	\begin{cases}
				\{(1,0,0)\}, & q \equiv 1 \pmod{3}, \\
				\{(1,0,0), (0,1,0)\}, & q \equiv -1 \pmod{3}, \\
				\left\{(\alpha, 1, 0) \mid \alpha \in \F_q, \alpha \in \nsq_{q} \cup \{0\} \right\} \cup \{(1,0,0)\}, & q \equiv 0 \pmod{3}. 
				\end{cases}  \\
\end{align*}
Let $q\geq 29$ be odd and $s \in \nsq_q$. Let
\begin{align*}
	& \cD_2: X_2^2(X_3 - \lambda X_2) - X_1 (s X_1 - X_3)^2 = 0, \quad \xi \in \F_q, \; {\mbox s.t. } \; \xi^3 + 3 \lambda \xi^2 + 3 s \xi + \lambda s = 0, \mbox{ then}\\
	& \cA_{\cD_2} = 	\begin{cases}
				\left\{\left(1, -\frac{8s \xi}{3 \xi^2 + s}, \frac{3s(\xi^2+3s)}{3\xi^2 + s}\right), (1, 0, s) \right\}, & q \not\equiv -1 \pmod{3}, \\
				\left\{ \left(1, \frac{8 \xi}{\xi^2-1}, \frac{3(9 - \xi^2)}{\xi^2 - 1} \right), \left(1, \frac{(\xi \mp 3)(1 \pm \xi)}{1 \mp \xi}, \frac{3\xi (\xi \pm 3)}{1 \mp \xi} \right), (1, 0, -3) \right\}, & q \equiv -1 \pmod{3} \mbox{ and } \\
				& s = -3. 
				\end{cases}  \\
\end{align*}
Let $q \equiv -1 \pmod{3}$ be odd, $q\geq 29$, and let
\begin{align*}
	& \cD_3: X_2^2(X_3 - \lambda X_2) - X_1 (3 X_1 + X_3)^2 = 0, \quad F(T) = T^3 + 3 \lambda T^2 - 9 T - 3 \lambda \\
	& \mbox{ irreducible over } \F_q, \mbox{ then } \cA_{\cD_3} = \{(1, 0, -3)\}. \\
\end{align*}
Let $q\equiv -1\pmod 3$ be odd, $q\geq 29$, and let
\begin{align*}
	& \cD_4: X_2^3 - 27(\lambda-1)^3 X_1^2X_3 - (3\lambda^2-3\lambda+1) X_1X_3^2 - 9(\lambda-1)(2\lambda^2-2\lambda+1) X_1X_2X_3 \\
	& \quad - \lambda (3\lambda^2-3\lambda+1) X_2^2X_3 = 0, \quad \lambda \in \F_q \setminus \left\{1, 1/2\right\}, \mbox{ then} \\		
	& \cA_{\cD_4} = 	\left\{ \left( (3\lambda^2-3\lambda+1)^2, -9(\lambda-1)^2(3\lambda^2-3\lambda+1), 27(\lambda-1)^3 \right), (1, -3, 0), \right. \\
	& \qquad \quad \left. \left(\lambda^2, -3(\lambda^2-\lambda+1), 27(\lambda-1)\right), \left(\lambda^2, -3(\lambda-1)^2, 0\right) \right\}. \\
\end{align*}
Let $q\geq 32$ be even and $\delta \in \F_q$, with $\Tr(\delta) = 1$. If 
\begin{align*}
	& \cD_5: X_2^3 + \delta (\delta+1) X_1^2 X_2 + \delta X_1^2 X_3 + X_1 X_3^2 + X_1 X_2 X_3 = 0, \mbox{ then}\\		
	& \cA_{\cD_5} = 	\begin{cases}
				\left\{ (1, \delta, \delta+1), (1, \delta, \delta) \right\}, & q \equiv 1 \pmod{3},\\
				\left\{ (0, 1, b), (0, 1, b+1), (1, \delta, \delta+1), (1, \delta, \delta) \right\}, & q \equiv -1 \pmod{3}, b \in \F_q, \\ 
				& b^2+b+\delta+1 = 0. 
				\end{cases}  \\
\end{align*}
Let $q\geq 32$ be even, and let
\begin{align*}
	& \cD_6: (\delta^2 + \delta + \lambda) X_1^3 + (\lambda+1) X_2^2 + (\delta + \lambda) X_1^2 X_2 + \lambda X_1 X_2^2 +  X_1 X_3^2 + X_1 X_2 X_3 + X_2^2 X_3 = 0, \\
	& \lambda \in \F_q, then \\		
	& \cA_{\cD_6} = 	\begin{cases} 
				\left\{ \left(1, \frac{\xi^2+\xi+1}{\xi(\xi+1)}, \frac{\xi^2 \lambda + \xi \lambda + 1}{\xi^2 (\xi^2+1)} \right), \left( 1, \frac{\xi^2+\xi+1}{\xi}, \frac{(\xi^2+1)(\xi^2+ \xi \lambda +1)}{\xi^2} \right), \left( 1, \frac{\xi^2+\xi+1}{\xi+1}, \frac{\xi^2(\xi^2+\xi \lambda + \lambda)}{\xi^2+1} \right), (1, 1, 0) \right\}, \\
				\qquad \qquad \quad q \equiv -1 \pmod{3},\  \delta = 1, \xi \in \F_q, \xi^3+ \lambda \xi^2 + (\lambda+1) \xi + 1 = 0, \\\\
				\left\{ (1, 1, 0) \right\}, \; q \equiv -1 \pmod{3}, \delta = 1,  F(T) = T^3 + \lambda T^2 + (\lambda+1) T + 1 \\
				\qquad \qquad \quad \mbox{ irreducible over } \F_q,\\\\ 
				\left\{ \left( \xi^2+\xi+\delta+1, \xi^2+\xi+\delta, (\delta+1)\xi^2 + \delta \xi + \delta +1\right), (1, 1, \delta+1) \right\},  \\
				\qquad \qquad \quad q \equiv 1 \pmod{3}, \xi \in \F_q, \xi^3 + (\lambda+1) \xi^2 + (\delta+\lambda+1) \xi + \delta \lambda + \lambda + 1 = 0. 
				\end{cases}
				\\
\end{align*}
Let $q\equiv -1\pmod 3$, $q\geq 32$ be even, and let
\begin{align*}
	& \cD_7: X_2^3 + X_1X_3^2 + \left(\frac{\lambda + 1}{\lambda}\right)^3 X_1^2X_3 + \frac{\lambda+1}{\lambda} X_1X_2X_3 = 0, \quad \lambda \in \F_q \setminus \left\{0, 1\right\},\  then \\	
	& \cA_{\cD_7} = 	\left\{ (0, 1, 0), \left( 1, \left( \frac{\lambda+1}{\lambda}  \right)^2, 0 \right), \left( 0, 1, \frac{\lambda+1}{\lambda} \right), \left( 1, \left( \frac{\lambda+1}{\lambda} \right)^2, \left( \frac{\lambda+1}{\lambda} \right)^3 \right)			\right\}. \\
\end{align*}
\end{lemma}
\begin{proof}
The plane curve $\cD_i$, $i=1,\ldots,7$, is singular and a line through its singular point has at most one further point in common with $\cD_i$. Then the singular point of $\cD_i$ belongs to $\cA_{\cD_i}$.

The curve $\cD_1: X_1 X_3^2 - X_2^3 = 0$ consists of $q+1$ points. The point $(1,0,0)$ is a cusp and $\cD_1$ has either one or $q$ inflexion points, according as $q \not\equiv 0 \pmod{3}$ or $q \equiv 0 \pmod{3}$. By \cite[Proposition 2.1]{BMP} the number of points of $\PG(2, q) \setminus \cD_1$ lying on no line intersecting $\cD_1$ in three points is either zero or one or $(q+1)/2$,  according as $q \equiv 1$ or $-1$ or $0 \pmod{3}$. In particular, if $q \equiv -1 \pmod{3}$ the unique point is $(0,0,1)$ and if $q \equiv 0 \pmod{3}$ the $(q+1)/2$ points are $(\alpha, 1, 0)$, where $\alpha$ is zero or belongs to $\nsq_q$.

\bigskip

\underline{Assume $q$ is odd}. Let $s \in \nsq_q$ and let $\lambda \in \F_q$ such that $F(T) = T^3 + 3 \lambda T^2 + 3 s T + \lambda s$ is reducible over $\F_q$. The polynomial $F$ has exactly one root in $\F_q$ whenever $q \not\equiv -1 \pmod{3}$; otherwise $F$ is completely reducible if and only if $(\lambda - \i)^{q-1}$ is a cube in $\F_{q^2}$, with $\i \in \F_{q^2}$ such that $\i^2 = s$. In this case the three roots of $F$ are in $\F_q$, see \cite{H1}. Let $\xi \in \F_q$ such that $F(\xi) = 0$. 

\bigskip

The curve $\cD_2: X_2^2(X_3 - \lambda X_2) - X_1 (s X_1 - X_3)^2 = 0$ has $q+2$ points. The point $(1,0,s)$ is an isolated double point and $\cD_2$ has either one or three inflexion points as $q \not\equiv -1 \pmod{3}$ or $q \equiv -1 \pmod{3}$. The projectivity of $\PG(2, q)$ associated with the matrix
\[
\begin{pmatrix}
-s \xi & s & \xi \\
\frac{3 \xi^2}{\xi^2-s} & \frac{3 \xi}{\xi^2 - s} & \frac{1}{\xi^2 - s} \\
-s & \xi & 1 
\end{pmatrix}
\]
maps $\cD_2$ to the cubic curve given by $X_2(X_1^2 - s X_3^2) = X_3^3$. By \cite[Proposition 2.3, Proposition 2.7]{BMP}, when $q\geq 29$, the number of points of $\PG(2, q) \setminus \cD_2$ lying on no line intersecting $\cD_2$ in three points is either one if $q \not\equiv -1 \pmod{3}$ or three if $q \equiv -1 \pmod{3}$. In particular, if $q \not\equiv -1 \pmod{3}$ the unique point is $\left(1, -\frac{8s \xi}{3 \xi^2 + s}, \frac{3s(\xi^2+3s)}{3\xi^2 + s}\right)$ and if $q \equiv -1 \pmod{3}$ the three points are $\left(1, \frac{8 \xi}{\xi^2-1}, \frac{3(9 - \xi^2)}{\xi^2 - 1} \right)$, $\left(1, \frac{(\xi - 3)(1 + \xi)}{1 - \xi}, \frac{3\xi (\xi + 3)}{1 - \xi} \right)$, $\left(1, \frac{(\xi + 3)(1 - \xi)}{1 + \xi}, \frac{3\xi (\xi - 3)}{1 + \xi} \right)$, where $s = -3$.

\bigskip
If $F$ is irreducible over $\F_q$, then $q \equiv -1 \pmod{3}$ and hence we may assume $s = -3$. In this case the projectivity of $\PG(2, q)$ associated with the matrix
\[
\begin{pmatrix}
-9 & 3 & -3 \\
\frac{9}{4} & \frac{2 \lambda + 3}{4} & \frac{1}{12} \\
3 & 3 & 1 
\end{pmatrix}
\]
maps $\cD_3$ to the cubic curve given by $X_2(X_1^2 + 3X_3^2) = X_3^3 + \frac{\lambda}{36} X_1 (X_1^2 - 9X_3^2)$. By \cite[Proposition 2.4]{BMP}, if $q\geq 29$, every point of $\PG(2, q) \setminus \cD_3$ lies on at least a line meeting $\cD_3$ in three points.

\bigskip

Let $\lambda \in \F_q \setminus \{1, 1/2\}$ and $\cD_4: X_2^3 - 27(\lambda-1)^3 X_1^2X_3 - (3\lambda^2-3\lambda+1) X_1X_3^2 - 9(\lambda-1)(2\lambda^2-2\lambda+1) X_1X_2X_3 - \lambda (3\lambda^2-3\lambda+1) X_2^2X_3 = 0$. The projectivity of $\PG(2,q)$ associated with the matrix 
\[
\begin{pmatrix}
-27(\lambda - 1)^3 & -9(\lambda-1)^3 & (3\lambda-2)(3\lambda^2-3\lambda+1) \\
\frac{27}{4} (\lambda-1)^3 & \frac{9}{4} \lambda^2 (\lambda-1) & \frac{\lambda^3}{4} \\
27 (\lambda-1)^3 & 3(\lambda-1)(3\lambda^2-2\lambda+1) & \lambda (3\lambda^2 - 3\lambda + 1) 
\end{pmatrix}
\] 
maps $\cD_4$ to the cubic curve given by $X_2(X_1^2 + 3X_3^2) = X_3^3$. In this case, from \cite[Proposition 2.3]{BMP}, when $q\geq 29$, apart from the singular point $\left( (3\lambda^2-3\lambda+1)^2, -9(\lambda-1)^2(3\lambda^2-3\lambda+1), 27(\lambda-1)^3 \right)$, $\cA_{\cD_4}$ contains the points $(1, -3, 0)$, $\left(\lambda^2, -3(\lambda^2-\lambda+1), 27(\lambda-1)\right)$, $\left(\lambda^2, -3(\lambda-1)^2, 0\right)$.

\bigskip

\underline{Assume $q$ is even}. Let $\delta \in \F_q$ such that $\Tr(\delta) = 1$. The curve $\cD_5: X_2^3 + \delta (\delta+1) X_1^2 X_2 + \delta X_1^2 X_3 + X_1 X_3^2 + X_1 X_2 X_3 = 0$ has $q+2$ points and $(1, \delta, \delta)$ is an isolated double point. If $q \equiv -1 \pmod{3}$, let $b \in \F_q$ such that $b^2+b+\delta+1 = 0$. The projectivity associated with the matrix
\[
\begin{pmatrix}
\delta (b+1) & b & 1 \\
\delta + 1 & 1 & 0 \\
\delta & 1 & 0  
\end{pmatrix}
\] 
maps $\cD_5$ to the cubic curve $X_2 (X_1^2+X_1X_3+X_3^2)+X_1^2X_3+X_1X_3^2 = 0$. By \cite[Proposition 2.5]{BMP}, when $q\geq 32$, apart from the singular point, $\cA_{\cD_5}$ consists of the points $(0, 1, b)$, $(0, 1, b+1)$, $(1, \delta, \delta + 1)$. If $q \equiv 1 \pmod{3}$, then the projectivity associated with the matrix 
\[
\begin{pmatrix}
\delta^2 & \delta & 0 \\
\delta^2 (\delta+1) & \delta^2 & 0 \\
0 & 1 & 1 
\end{pmatrix}
\]
sends $\cD_5$ to the cubic curve $X_2 (X_1^2+X_1X_3+ \delta X_3^2) + (\delta+1) X_1^3 + \delta X_1^2X_3 + \delta^2 X_1X_3^2 = 0$ and by \cite[Proposition 2.9]{BMP}, when $q\geq 64$, $\cA_{\cD}$ contains the isolated double point and the point $(1, \delta, \delta+1)$.

Set $\cD_6: (\delta^2 + \delta + \lambda) X_1^3 + (\lambda+1) X_2^2 + (\delta + \lambda) X_1^2 X_2 + \lambda X_1 X_2^2 +  X_1 X_3^2 + X_1 X_2 X_3 + X_2^2 X_3 = 0$, for some $\lambda \in \F_q$. If $q \equiv 1 \pmod{3}$, let $\xi \in \F_q$ such that $\xi^3 + (\lambda+1) \xi^2 + (\delta+\lambda+1) \xi + \delta \lambda + \lambda + 1 = 0$. The projectivity of $\PG(2,q)$ associated with the matrix 
\[
\begin{pmatrix}
\xi^2+\xi(\delta + 1) & \xi^2 & \xi \\
\frac{\xi \delta (\xi^3 + (\delta+1) \xi^2 + \xi + \delta^2 + \delta + 1)}{\xi^2+\xi+\delta} & \frac{\xi \delta (\xi^3+\xi^2+(\delta + 1) \xi +1)}{\xi^2+\xi+\delta} & \frac{\xi \delta (\xi^2+\xi+\delta + 1)}{\xi^2+\xi+\delta} \\
\xi^2 + \xi & \frac{\xi^2 + \xi \delta}{\delta} & \frac{\xi^2}{\delta} 
\end{pmatrix}
\]  
sends $\cD_6$ to the cubic curve $X_2 (X_1^2+X_1X_3+ \delta X_3^2) + (\delta+1) X_1^3 + \delta X_1^2X_3 + \delta^2 X_1X_3^2 = 0$. From \cite[Proposition 2.9]{BMP}, when $q\geq 64$, it follows that $\cA_{\cD_6}$ consists of the points $(1, 1, \delta+1)$ and  $\left( \xi^2+\xi+\delta+1, \xi^2+\xi+\delta, (\delta+1)\xi^2 + \delta \xi + \delta +1\right)$. Let $q \equiv -1 \pmod{3}$ and let $\delta = 1$. If $F(T) = T^3 + \lambda T^2 + (\lambda + 1) T + 1$ is reducible over $\F_q$, let $\xi \in \F_q$ such that $F(\xi) = 0$. Thus the projectivity of $\PG(2,q)$ associated with the matrix 
\[
\begin{pmatrix}
\frac{\xi^2 \lambda + \xi \lambda + 1}{\xi(\xi+1)} & \frac{\xi^2 \lambda + \xi \lambda + 1}{\xi (\xi +1)} & 1 \\
\frac{\xi^2 (\lambda+1) + \xi \lambda + \xi}{\xi^2+\xi+1} & \frac{\xi^2 \lambda + \xi \lambda}{\xi^2+\xi+1} & \frac{\xi^2+\xi}{\xi^2 + \xi + 1} \\
\frac{\xi^3 + \xi^2 \lambda + \xi \lambda}{\xi + 1} & \frac{\xi^2 + \xi^2 \lambda + \xi \lambda}{\xi + 1} & \xi
\end{pmatrix}
\]
maps $\cD_6$ to the cubic curve $X_2 (X_1^2+X_1X_3+X_3^2) + X_1^2X_3 + X_1X_3^2 = 0$ and, by \cite[Proposition 2.5]{BMP}, when $q\geq 32$, $\cA_{\cD_6}$ consists of the points $\left(1, \frac{\xi^2+\xi+1}{\xi(\xi+1)}, \frac{\xi^2 \lambda + \xi \lambda + 1}{\xi^2 (\xi^2+1)} \right)$, $\left( 1, \frac{\xi^2+\xi+1}{\xi}, \frac{(\xi^2+1)(\xi^2+ \xi \lambda +1)}{\xi^2} \right)$, $\left( 1, \frac{\xi^2+\xi+1}{\xi+1}, \frac{\xi^2(\xi^2+\xi \lambda + \lambda)}{\xi^2+1} \right)$, $(1, 1, 0)$. Assume $F(T)$ is irreducible over $\F_q$. The projectivity of $\PG(2,q)$ associated with the matrix
\[
\begin{pmatrix}
\lambda & \lambda & 0 \\
\lambda+1 & \lambda & 1 \\
0 & 0 & \lambda 
\end{pmatrix}
\]
sends $\cD_6$ to the cubic curve $X_2 (X_1^2+X_1X_3+X_3^2) + X_1^3 + X_1^2X_3 + X_3^3 + \frac{\lambda + 1}{\lambda} \left( X_1^3 + X_1X_3^2 + X_3^3 \right) = 0$, if $\lambda \ne 0$, whereas $\cD_6$ is sent to the cubic curve $X_2 (X_1^2+X_1X_3+X_3^2) + X_1^3 + X_1^2X_3 + X_3^3 = 0$ by the projectivity associated with the matrix  
\[
\begin{pmatrix}
1 & 1 & 1 \\
1 & 0 & 1 \\
0 & 0 & 1 
\end{pmatrix}
\]
if $\lambda = 0$. In this case, by \cite[Proposition 2.8]{BMP}, when $q\geq 32$, we get $\cA_{\cD_6} = \left\{ (1, 1, 0) \right\}$.

\bigskip

Finally, let $\lambda \in \F_q \setminus \{0, 1\}$ and $\cD_7: X_2^3 + X_1X_3^2 + \left(\frac{\lambda + 1}{\lambda}\right)^3 X_1^2X_3 + \frac{\lambda+1}{\lambda} X_1X_2X_3 = 0$. The projectivity of $\PG(2,q)$ associated with the matrix
\[
\begin{pmatrix}
\left(\frac{\lambda + 1}{\lambda}\right)^3 & \frac{\lambda + 1}{\lambda} & 0 \\
0 & \frac{\lambda+1}{\lambda} & 0 \\
0 & \frac{\lambda+1}{\lambda} & 1 
\end{pmatrix},
\] 
maps $\cD_7$ to the cubic curve given by $X_2(X_1^2 + X_1 X_3 + X_3^2) + X_1^2 X_3 + X_1 X_3^2 = 0$. In this case, from \cite[Proposition 2.5]{BMP}, when $q\geq 32$, apart from the singular point $\left( 1, \left( \frac{\lambda+1}{\lambda} \right)^2, \left( \frac{\lambda+1}{\lambda} \right)^3 \right)$, $\cA_{\cD_7}$ contains the points $(0, 1, 0)$, $\left( 1, \left( \frac{\lambda+1}{\lambda}  \right)^2, 0 \right)$, $\left( 0, 1, \frac{\lambda+1}{\lambda} \right)$.
\end{proof}

\subsection{The NMDS-sets}

Let $\cC$ be the twisted cubic of $\PG(3, q)$ consisting of the $q+1$ points $\{P_t \mid t \in \F_q\} \cup \{(0,0,0,1)\}$, where $P_t = (1,t,t^2,t^3)$. It is well known that a line of $\PG(3, q)$ meets $\cC$ in at most $2$ points and a plane shares with $\cC$ at most $3$ points (i.e., $\cC$ is a so called {\em $q+1$-arc}). A line of $\PG(3,q)$ joining two distinct points of $\cC$ is called a {\em real chord} and there are $q(q+1)/2$ of them. Let $\bar{\cC} = \{P_t \mid t \in \F_{q^2}\} \cup \{(0,0,0,1)\}$ be the twisted cubic of $\PG(3, q^2)$ which extends $\cC$ over $\F_{q^2}$. The line of $\PG(3, q^2)$ obtained by joining $P_t$ and $P_{t^q}$, with $t\notin \F_q$, meets the canonical Baer subgeometry $\PG(3, q)$ in the $q+1$ points of a line skew to $\cC$. Such a line is called {\em imaginary chord} and they are $q(q-1)/2$ in number. Also, for each point $P$ of $\cC$, the line $\ell_{P} = \langle P, P' \rangle$, where $P'$ equals $(0, 1, 2t, 3t^2)$ or $U_3$ if $P = P_t$ or $P = U_4$, respectively, is called the {\em tangent} line to $\cC$ at $P$. At each point $P_t$ (resp. $U_4$) of $\cC$ there corresponds the {\em osculating plane} with equation $t^3 X_1 - 3t^2 X_2 + 3t X_3 - X_4 = 0$ (resp. $X_1 = 0$), meeting $\cC$ only at $P_t$ (resp. $U_4$) and containing the tangent line. For more properties and results on $\cC$ the reader is referred to \cite{H2}.  

\begin{lemma}\cite[Theorem 21.1.9]{H2}
Every point of $\PG(3, q) \setminus \cC$ lies on exactly one chord or a tangent of $\cC$.
\end{lemma}

Let $G$ be the group of projectivities of $\PG(3, q)$ stabilizing $\cC$. Then $G \simeq \PGL(2, q)$ whenever $q \ge 5$, and elements of $G$ are induced by the matrices
\begin{equation} \label{mat}
\begin{pmatrix}
a^3 & 3 a^2 b & 3 a b^2 & b^3 \\
a^2 c & a^2 d + 2 abc & b^2 c + 2 abd & b^2 d \\
a c^2 & bc^2 + 2 acd & ad^2 + 2 bcd & b d^2 \\
c^3 & 3 c^2 d & 3 c d^2 & d^3 \\
\end{pmatrix},
\end{equation}
where $a,b,c,d \in \F_q$, $ad-bc \ne 0$.

\begin{lemma}\cite[Corollary 5, Lemma 21.1.11]{H2}
The group $G$ has one or two orbits on points lying on imaginary chords of $\cC$ according as $q \not\equiv -1 \pmod{3}$ or $q \equiv -1\pmod{3}$, respectively. $G$ has one or two orbits on points of $\PG(3, q) \setminus \cC$ lying on tangent lines to $\cC$ according as $q \not\equiv 0 \pmod{3}$ or $q \equiv 0 \pmod{3}$, respectively.
\end{lemma}

As representatives of $G$--orbits on points lying on tangent lines and not on $\cC$, we may consider either $U_2$ or $U_1 - U_2$ and $U_2$, according as $q \not\equiv 0 \pmod{3}$ or $q \equiv 0 \pmod{3}$. By \cite[Lemma 21.1.11]{H2}, if $q \equiv -1 \pmod{3}$, a point lying on an imaginary chord belongs to one of the two orbits according  as there are three osculating planes passing through it or none, respectively.

\bigskip

Let $q$ be odd. Let $s$ be a fixed element of $\nsq_q$ and let $\ell$ be the line joining $Q = (1,0,s,0)$ and $Q_1 = (0,1,0,s)$. The line $\ell$ is an imaginary chord, whose extension over $\F_{q^2}$ intersects $\bar{\cC}$ in the conjugated points $P_{\i}$ and $P_{\i^q}$, where $\i \in \F_{q^2}\setminus\F_q$ such that $\i^2 = s$. If $q \not\equiv -1 \pmod{3}$, the point $Q$ belongs to the unique $G$--orbit of points lying on imaginary chords. If $q \equiv -1 \pmod{3}$, we may assume $s = -3$. In this case through $Q$ and $Q_1$ there pass three osculating planes, whereas if $\tilde F(T) = T^3 - 3 \lambda T^2 - 9 T + 3 \lambda$ is irreducible over $\F_q$ there is no osculating plane through the point $(1, \lambda, s, \lambda s)$. Note that $\tilde{F}(T)$ is irreducible over $\F_q$ if and only if $(\lambda + \i)^{q-1}$ is not a cube in $\F_{q^2}$. This happens if and only if $F(T) = T^3 + 3 \lambda T^2 - 9 T - 3\lambda$ is irreducible over $\F_q$. 


Let $q$ be even. Let $\delta$ be a fixed element of $\F_q$ such that $\Tr(\delta) = 1$ and let $\ell$ be the line joining $S = (1,0,\delta,\delta)$ and $S_1 = (0,1,1,\delta+1)$. The line $\ell$ is an imaginary chord, whose extension over $\F_{q^2}$ meets $\bar{\cC}$ in the points $P_{\i}$ and $P_{\i^q}$, with $\i \in \F_{q^2}\setminus\F_q$ such that $\i^2 + \i = \delta$. Therefore, if $q \not\equiv -1 \pmod{3}$, the point $S_1$ is a representative of the unique $G$--orbit of points lying on imaginary chords. If $q \equiv -1 \pmod{3}$, we may assume $\delta = 1$. In this case through $S_1$ there pass three osculating planes; through the point $(1, \lambda, \lambda + 1, 1)$ there are three osculating planes or none, according as $\tilde F(T) = T^3 + \lambda T^2 + (\lambda + 1) T + 1$ is reducible or not over $\F_q$. Note that $\tilde F(T)$ is irreducible over $\F_q$ if and only if $(\lambda + \i)^{1-q}$ is not a cube in $\F_{q^2}$.

\bigskip
Let $\cX$ be the pointset obtained by adding to $\cC$ a point on a tangent line to $\cC$. We show that $\cX$ is an NMDS-set that either is complete or it can be completed by adding at most one further point.

\begin{prop}
The set $\cX = \cC \cup \{U_2\}$ is an NMDS-set such that 
\begin{enumerate}
\item $\cX$ is complete if $q \equiv 1 \pmod{3}$ and $q\geq 23$;
\item $\cX \cup \{R\}$ is complete, where $R \in \left\{(\alpha, \beta, 1, 0) \mid \alpha, \beta \in \F_q, \alpha \in \nsq_q \cup \{0\} \right\}$ if $q \equiv 0 \pmod{3}$ and $q\geq 81$, or $R \in U_2 U_3 \setminus \{U_2\}$ if $q \equiv -1 \pmod{3}$ and $q\geq 23$. 
\end{enumerate}
\end{prop}
\begin{proof}
By projecting $\cC$ from $U_2$ onto the plane $\pi: X_2 = 0$, we obtain the $q+1$ points of the cubic curve $\cD_1: X_1 X_4^2 - X_3^3 = 0$ of $\pi$. The set $\cX = \cC \cup \{U_2\}$ consists of $q+2$ points no three collinear and no five coplanar. Indeed if $r$ is a line of $\pi$ meeting $\cD_1$ in three points, then the plane spanned by $r$ and $U_2$ contains four points of $\cX$. If a point $R$ of $\pi \setminus \cD_1$ lies on a line of $\pi$ intersecting $\cD_1$ in three points, then no point of the line $U_2 R$ can be added to $\cX$ in order to get a larger NMDS-set. 

By Lemma \ref{plane}, if $q \equiv 1 \pmod{3}$ and $q\geq 23$, we have that $\cA_{\cD_1} = \{U_1\}$ and hence $\cX$ is complete. If $q \equiv -1 \pmod{3}$  with $q\geq 23$, then $\cA_{\cD_1} = \{U_1, U_3\}$. In this case no point of $U_2 U_3$ is on a real chord and the result follows. If $q \equiv 0 \pmod{3}$ and $q\geq 81$, there are $(q+1)/2$ points of $\pi \setminus \cD_1$ lying on no line intersecting $\cD_1$ in three points. They are $(\alpha, 0, 1, 0)$, where $\alpha$ is zero or in $\nsq_q$ and they lie on the line $r: X_4 = 0$ of $\pi$. It can be easily checked that there arise $q(q+1)/2$ points of the plane $\langle r, U_2 \rangle$ none of them on a real chord. Hence each of them can be added to $\cX$ in order to get a larger NMDS-set. These $q(q+1)/2$ points are permuted into two orbits under the action of the stabilizer of $U_2$ in $G$, namely $\left\{(\alpha, \beta, 1, 0) \mid \alpha, \beta \in \F_q, \alpha \in \nsq_q\right\}$ and $U_2 U_3 \setminus \{U_2\}$. Let $R$ be a point belonging to one of these two orbits. 

If $R = U_3$, by projecting $\cC$ from $U_3$ onto the plane $\pi': X_3 = 0$, we obtain the $q+1$ points of the cubic curve $\cD': X_1^2 X_4 - X_2^3 = 0$ (which is, up to a projectivity, the cubic $\cD_1$ of Lemma \ref{plane}) of $\pi'$. Hence, if there were another point, say $R'$, such that $\cC \cup \{R, R'\}$ is an NMDS-set, then $R' \in \left\{(0, 1, \beta', \alpha') \mid \alpha', \beta' \in \F_q, \alpha' \in \nsq_q \cup \{0\}\right\}$ and it would belong to a line joining $U_2$ and a point $T=(\alpha,0,1,0)$ for some $\alpha\in\nsq_q$. This is a contradiction since the two lines $U_3R'$ and $U_2T$ are skew. Hence $\cX \cup \{U_3\}$ is complete.

Let $R = (1, 0, s, 0)$. By projecting $\cC$ from $R$ onto the plane $\pi'': X_1 = 0$, we obtain the points of the plane cubic curve $\cD'': X_2(sX_2 - X_4)^2 - X_3^2 X_4 = 0$ of $\pi''$ (which is, up to a projectivity, the cubic $\cD_2$ of Lemma \ref{plane} with $\lambda=0$) consisting of the isolated double point $Q_1=(0,1,0,s)$ and of $q+1$ simple points. In this case $\cA_{\cD''} = \{U_2, Q_1\}$. Therefore if there were another point $R'$ such that $\cC \cup \{R, R'\}$ is an NMDS-set, then $R'$ would belong to the intersection between the two lines $ \{(\lambda, 1, \lambda s, s) \mid \lambda \in \F_q\}$ and $U_2T$, where $T=(\alpha,0,1,s)$, for some $\alpha\in\nsq_q$. As before, it gives a contradiction since the two lines are skew. Hence $\cX \cup \{R\}$ is complete.
\end{proof}

\begin{prop}
Let $q \ge 81$, $q \equiv 0 \pmod{3}$. The set $\cX = \cC \cup \{U_1 - U_2\}$ is a complete NMDS-set. 
\end{prop}
\begin{proof}
By projecting $\cC$ from $U_1 - U_2$ onto the plane $\pi: X_1 = 0$, we obtain the $q+1$ points of the cubic curve $\cD: X_2 X_4^2 - X_3^2 X_4 - X_3^3 = 0$ of $\pi$. The curve $\cD$ has a cusp, namely $U_2$, and no inflexion points. The set $\cX = \cC \cup\{U_1 - U_2\}$ consists of $q+2$ points no three collinear and no five coplanar.  By \cite[Proposition 2.6]{BMP} every point of $\pi \setminus \cD$ lies on a line of $\pi$ intersecting $\cD$ in three points. Hence $\cX$ is complete.
\end{proof}

Next we deal with the case when $\cX$ is obtained by adding to $\cC$ a point on an imaginary chord. It turns out that $\cX$ is an NMDS-set that is either complete or it can be completed by adding at most three further points. 

\begin{prop}
Let $q \ge 29$ be odd and let $Q = (1, 0, s, 0)$, $Q_1 = (0, 1, 0, s)$, $Q_2 = (0, 1, 0, 9s)$, $Q_5 = (0, 1, 3, 0)$ and $Q_6 = (0, 1, -3, 0)$. The set $\cX = \cC \cup \{Q\}$ is an NMDS-set such that 
\begin{enumerate}
\item $\cX \cup \{R\}$ is complete, where $R \in Q Q_1 \setminus \{Q\}$ or $R \in Q Q_2 \setminus \{Q\}$, if $q \not\equiv -1 \pmod{3}$;
\item either $\cX \cup \{R\}$, $R \in Q Q_1 \setminus \{Q\}$, is complete, or $\cX \cup \{R_1, R_2, R_3\}$ is complete, where $R_1, R_2, R_3$ are three distinguished points belonging to $Q Q_2 \setminus \{Q\}$, $Q Q_5 \setminus \{Q\}$, $Q Q_6 \setminus \{Q\}$, respectively, if $q \equiv -1 \pmod{3}$.
\end{enumerate}
\end{prop}
\begin{proof}
By projecting $\cC$ from $Q$ onto the plane $\pi: X_1 = 0$, we obtain the $q+1$ points of the cubic curve $\cD: X_3^2 X_4 - X_2 (s X_2 - X_4)^2 = 0$ of $\pi$ (which is, up to projectivities, the plane cubic curve $\cD_2$ with $\lambda=0$ and $\xi=0$ of Lemma \ref{plane}). As before, if a point of the line $Q R$, where $R \in \pi$, can be added to $\cX$ in order to get a larger NMDS-set then no line of $\pi$ meeting $\cD$ in three points passes through $R$.

Assume $q \not\equiv -1 \pmod{3}$, and $q\geq 29$. By Lemma \ref{plane}, the set $\cA_{\cD}$ consists of the points $Q_1 = (0, 1, 0, s)$ and $Q_2 = (0, 1, 0, 9s)$. Note that every point of $Q Q_2$ lies on an imaginary chord, except $Q_2$ if $q \equiv 0 \pmod{3}$. Hence $R \in (Q Q_1 \cup Q Q_2) \setminus \{Q\}$ if and only if $\cX \cup \{R\}$ is an NMDS-set, for some $R \notin \cX$. Let $R \in Q Q_1 \setminus \{Q\}$. For $R = Q_1$ by projecting $\cC$ from $Q_1$ onto the plane $\pi': X_2 = 0$, we get the $q+1$ simple points of the cubic curve $\cD': X_3(sX_1 - X_3)^2 - X_4^2 X_1 = 0$ of $\pi'$ which is mapped to $X_1(s X_1 - X_4)^2 - X_3^2 X_4 = 0$ by the projectivity $X_1' = X_3, X_3' = X_4, X_4' = s^2 X_1$ of $\pi'$ (which is again, up to projectivities, the cubic curve $\cD_2$ of Lemma \ref{plane}). Hence, by Lemma~\ref{plane} the two points of $\cA_{\cD'}$ are $Q_3 = (9, 0, s, 0)$ and $Q$. As before we have that $R' \in (Q Q_1 \cup Q_1 Q_3) \setminus \{Q_1\}$ if and only if $\cC \cup \{Q_1, R'\}$ is an NMDS-set, for some $R' \notin (\cC \cup \{Q_1\})$. It follows that $\cX \cup \{Q_1\}$ is complete. Similarly, for $R = (1, \lambda, s, \lambda s)$, $\lambda \in \F_q \setminus \{0\}$, by projecting $\cC$ from $R$ onto $\pi$, we get the $q+1$ simple points of the cubic curve $\cD'': X_3^2(X_4 - \lambda X_3) - X_2 (s X_2 - X_4)^2 = 0$ (which is again, up to projectivities, the cubic curve $\cD_2$ of Lemma \ref{plane}). In this case the two points of $\cA_{\cD''}$ are $Q_4 = \left(0, 1, -\frac{8s \xi}{3 \xi^2 + s}, \frac{3s(\xi^2+3s)}{3\xi^2 + s}\right)$, $Q_1$ and $R' \in (R Q_1 \cup R Q_4) \setminus \{R\}$ if and only if $\cC \cup \{R, R'\}$ is an NMDS-set, for some $R' \notin (\cC \cup \{R\})$. Again $\cX \cup \{R\}$ is complete. We have seen that if $R \in Q Q_1 \setminus \{Q\}$ then there is no point $R'$ of $Q Q_2 \setminus \{Q\}$ such that $\cX \cup \{R, R'\}$ is an NMDS-set. This implies that if $R \in Q Q_2 \setminus \{Q\}$ then there is no point $R'$ of $Q Q_1 \setminus \{Q\}$ such that $\cX \cup \{R, R'\}$ is an NMDS-set.

Assume $q \equiv -1 \pmod{3}$, $q\geq 29$ and $s = -3$. By Lemma \ref{plane}, the points of $\cA_{\cD}$ are $Q_1 = (0, 1, 0, -3)$, $Q_2 = (0, 1, 0, -27)$, $Q_5 = (0, 1, 3, 0)$ and $Q_6 = (0, 1, -3, 0)$. Some calculations show that the lines $Q Q_i$, $i = 2,5,6$, are permuted in a single orbit by the subgroup of $G$ generated by 
\[
\begin{pmatrix}
1 & 3 & 3 & 1 \\
-3 & -5 & -1 & 1 \\
9 & 3 & -5 & 1 \\
-27 & 27 & -9 & 1  
\end{pmatrix}.
\]
Such a subgroup has order three and fixes $Q$. Moreover the point $(\lambda, 1, 3(1 - \lambda), 0) \in Q Q_5$ is on an imaginary chord if $\lambda \in \F_q \setminus \{1, 1/2\}$, otherwise it lies on a tangent line. A point $R$ belongs to $(Q Q_1 \cup Q Q_2 \cup Q Q_5 \cup Q Q_6) \setminus \{Q\}$ if and only if $\cX \cup \{R\}$ is an NMDS-set, for some $R \notin \cX$. Let $R \in Q Q_1 \setminus \{Q\}$. Let $\cD'$ or $\cD''$ be the cubic curves obtained by projecting $\cC$ from $R$ onto the plane $\pi': X_2 = 0$ or $\pi$, according as $R$ equals $Q_1$ or $R = (1, \lambda, -3, -3 \lambda)$, $\lambda \in \F_q \setminus \{0\}$, respectively. Note that $\cD'$ is projectively equivalent to the cubic curve $\cD_2$ of Lemma \ref{plane}, whereas $\cD''$ is equivalent either to $\cD_3$ or $\cD_2$ of Lemma \ref{plane}, according as the polynomial $T^3+3\lambda T^2-9T-3\lambda$ is irreducible over $\F_q$ or not, respectively. By repeating the previous arguments we find that, when $q\geq 29$, $\cA_{\cD'}$ consists of the points $Q$, $(1,0,-1/3,0)$, $(0,0,1, \pm3)$, whereas $\cA_{\cD''}$ either consists of the point $Q_1$ or is formed by the points $Q_1$,  $\left(0, 1, \frac{8 \xi}{\xi^2-1}, \frac{3(9 - \xi^2)}{\xi^2 - 1} \right)$, $\left(0, 1, \frac{(\xi \mp 3)(1 \pm \xi)}{1 \mp \xi}, \frac{3\xi (\xi \pm 3)}{1 \mp \xi} \right)$. It follows that $\cX \cup \{R\}$ is complete.

Let $R \in Q Q_5 \setminus \{Q\}$. Let $\cD''': X_3^3 - 27(\lambda-1)^3 X_1^2X_4 - (3\lambda^2-3\lambda+1) X_1X_4^2 - 9(\lambda-1)(2\lambda^2-2\lambda+1) X_1X_3X_4 - \lambda (3\lambda^2-3\lambda+1) X_3^2X_4 = 0$ be the cubic curve obtained by projecting $\cC$ from $R$ onto the plane $\pi'$ (which is, up to projectivities, the cubic curve $\cD_4$ of Lemma \ref{plane}). By Lemma \ref{plane}, for $q\geq 29$, we have that $\cA_{\cD'''}$ consists of the four points: 
\begin{align*}
& \left(1, 0, -3, 0\right), \\ 
& \left(\lambda^2, 0, -3(\lambda-1)^2, 0\right), \\
& \left( (3\lambda^2-3\lambda+1)^2, 0, -9(\lambda-1)^2(3\lambda^2-3\lambda+1), 27(\lambda-1)^3 \right), \\
& \left(\lambda^2, 0, -3(\lambda^2-\lambda+1), 27(\lambda-1)\right). 
\end{align*}
It follows that $\left( \lambda, 1 - \lambda, -3\lambda, 27(\lambda-1) \right)$, $\left( \lambda, 1 - 2\lambda, 3(\lambda-1), 0 \right)$ are the only points that can be added to $\cX \cup \{R\}$ in order to have a complete NMDS-set.
\end{proof}

\begin{prop}
Let $q \ge 32$ even and let $S = (1, 0, \delta, \delta)$, $S_1 = (0, 1, 1, \delta+1)$, $S_2 = (1, 0, \delta, \delta+1)$, $S_4 = (0, 0, 1, 1)$ and $S_5 = (0,0,1,0)$. The set $\cX = \cC \cup \{S_1\}$ is an NMDS-set such that 
\begin{enumerate}
\item $\cX \cup \{R\}$ is complete, where $R \in S_1 S \setminus \{S_1\}$ or $R \in S_1 S_2 \setminus \{S_1\}$, if $q \equiv 1 \pmod{3}$;
\item either $\cX \cup \{R\}$, $R \in S_1 S \setminus \{S_1\}$, is complete, or $\cX \cup \{R_1, R_2, R_3\}$ is complete, where $R_1, R_2, R_3$ are three distinguished points belonging to $S_1 S_2 \setminus \{S_1\}$, $S_1 S_4 \setminus \{S_1\}$, $S_1 S_5 \setminus \{S_1\}$, respectively, if $q \equiv -1 \pmod{3}$.
\end{enumerate}
\end{prop}
\begin{proof}
By projecting $\cC$ from $S_1$ onto the plane $\pi: X_2 = 0$, we obtain the $q+1$ points of the cubic curve $\cD: X_3^3 + \delta (\delta+1) X_1^2 X_3 + \delta X_1^2 X_4 + X_1 X_4^2 + X_1 X_3 X_4 = 0$ of $\pi$ (which is, up to projectivities, the cubic curve $\cD_5$ of Lemma \ref{plane}). 

Assume $q \equiv 1 \pmod{3}$. By Lemma \ref{plane}, $\cA_{\cD} = \left\{ S, S_2 \right\}$. Observe that $\frac{1 + \lambda + (\delta + \lambda^2 + \lambda + 1) \i}{\delta + \lambda^2 + \lambda}$ is a root of $X^2 + \frac{\delta+\lambda^2+\lambda+1}{\delta+\lambda^2+\lambda} X + \frac{\delta^2+\delta \lambda^2 + \delta \lambda + \lambda}{\delta+\lambda^2+\lambda}$ and hence $\Tr\left( \frac{(\delta + \lambda^2 + \lambda)(\delta^2 + \delta \lambda^2 + \delta \lambda + \lambda)}{(\delta+\lambda^2+\lambda+1)^2} \right) = 1$. It follows that every point of $S_1 S_2$ lies on an imaginary chord and $R \in (S_1 S \cup S_1 S_2) \setminus \{S_1\}$ if and only if $\cX \cup \{R\}$ is an NMDS-set, for some $R \notin \cX$. Let $R = (1, \lambda, \delta + \lambda, \delta(1 + \lambda) + \lambda) \in S S_1 \setminus \{S_1\}$. By projecting $\cC$ from $R$ onto the plane $\pi': X_1 = 0$, we get the $q+1$ simple points of the cubic curve $\cD': (\delta^2 + \delta + \lambda) X_2^3 + (\lambda+1) X_3^2 + (\delta + \lambda) X_2^2 X_3 + \lambda X_2 X_3^2 +  X_2 X_4^2 + X_2 X_3 X_4 + X_3^2 X_4 = 0$ of $\pi'$ (which is, up to projectivities, the cubic curve $\cD_6$ of Lemma \ref{plane}). By Lemma~\ref{plane} the two points of $\cA_{\cD'}$ are $S_3 = \left(0, \xi^2+\xi+\delta+1, \xi^2+\xi+\delta, (\delta+1)\xi^2 + \delta \xi + \delta +1\right)$ and $S_1$. As before we have that $R' \in (R S_1 \cup R S_3) \setminus \{R\}$ if and only if $\cC \cup \{R, R'\}$ is an NMDS-set, for some $R' \notin (\cC \cup \{R\})$. It follows that $\cX \cup \{R\}$ is complete. We have seen that if $R \in S_1 S \setminus \{S_1\}$ then there is no point $R'$ of $S_1 S_2 \setminus \{S_1\}$ such that $\cX \cup \{R, R'\}$ is an NMDS-set. This implies that if $R \in S_2 S_1 \setminus \{S_1\}$ then there is no point $R'$ of $S_1 S \setminus \{S_1\}$ such that $\cX \cup \{R, R'\}$ is an NMDS-set.  

Assume $q \equiv -1 \pmod{3}$ and $\delta = 1$. By Lemma \ref{plane}, the points of $\cA_{\cD}$ are $S = (1, 0, 1, 1)$, $S_2 = (1, 0, 1, 0)$, $S_4 = (0, 0, 1, 1)$ and $S_5 = (0,0,1,0)$ (in such a case $\cD$ is projectively equivalent to $\cD_5$ with $\delta =b =1$ ). Some calculations show that the lines $S_1 S_i$, $i = 2,4,5$, are permuted in a single orbit by the subgroup of $G$ generated by 
\[
\begin{pmatrix}
0 & 0 & 0 & 1 \\
0 & 0 & 1 & 1 \\
0 & 1 & 0 & 1 \\
1 & 1 & 1 & 1  
\end{pmatrix}.
\]
Such a subgroup has order three and fixes $S_1$. The point $(0, \lambda, \lambda + 1, 0) \in S_1 S_5 \setminus \{S_1\}$ is on an imaginary chord if $\lambda \in \F_q \setminus \{0, 1\}$, otherwise it lies on a tangent line. A point $R$ belongs to $(S_1 S \cup S_1 S_2 \cup S_1 S_4 \cup S_1 S_5) \setminus \{S_1\}$ if and only if $\cX \cup \{R\}$ is an NMDS-set, for some $R \notin \cX$. Let $R \in S_1 S \setminus \{S_1\}$. Let $\cD'$ be the cubic curve obtained by projecting $\cC$ from $R$ onto the plane $\pi': X_1 = 0$ (which is again, up to projectivities, the cubic curve $\cD_6$ of Lemma \ref{plane}). By repeating the previous arguments and taking Lemma \ref{plane} into account, we find that, when $q\geq 32$, $\cA_{\cD'}$ either consists of the point $S_1$ or is formed by the points $\left(0, 1, \frac{\xi^2+\xi+1}{\xi(\xi+1)}, \frac{\xi^2 \lambda + \xi \lambda + 1}{\xi^2 (\xi^2+1)} \right)$, $\left(0, 1, \frac{\xi^2+\xi+1}{\xi}, \frac{(\xi^2+1)(\xi^2+ \xi \lambda +1)}{\xi^2} \right)$, $\left(0, 1, \frac{\xi^2+\xi+1}{\xi+1}, \frac{\xi^2(\xi^2+\xi \lambda + \lambda)}{\xi^2+1} \right)$, $S_1$, according as the polynomial $T^3+\lambda T^2+(\lambda+1)T+1$ is irreducible over $\F_q$ or not, respectively. It follows that $\cX \cup \{R\}$ is complete.

Let $R \in S_1 S_5 \setminus \{S_1\}$. Let $\cD''': X_3^3 + X_1X_4^2 + \left(\frac{\lambda + 1}{\lambda}\right)^3 X_1^2X_4 + \frac{\lambda+1}{\lambda} X_1X_3X_4 = 0$ be the cubic curve obtained by projecting $\cC$ from $R$ onto the plane $\pi$ (which is, up to projectivities, the cubic curve $\cD_7$ of Lemma \ref{plane}). By Lemma \ref{plane}, we have that $\cA_{\cD'''}$ consists of the four points: 
\begin{align*}
& \left( 1, 0, \left( \frac{\lambda+1}{\lambda} \right)^2, \left( \frac{\lambda+1}{\lambda} \right)^3 \right), (0, 0, 1, 0), \left( 1, 0, \left( \frac{\lambda+1}{\lambda}  \right)^2, 0 \right), \left( 0, 0, 1, \frac{\lambda+1}{\lambda} \right).
\end{align*}
It follows that $\left( \lambda, 1, \lambda + 1, 0 \right)$, $\left( 0, \lambda, 1, \lambda + 1 \right)$ are the only points that can be added to $\cX \cup \{R\}$ in order to have a complete NMDS-set.
\end{proof}

In a similar way, by taking into account the cubic curves $\cD_3$ and $\cD_6$ of Lemma \ref{plane}, it can be checked that the following result holds true.

\begin{prop}
Let $q \equiv -1 \pmod{3}$ and let $R$ be the point $(1, \lambda, -3, -3 \lambda)$, where $\lambda \in \F_q$ is such that $T^3-3\lambda T^2 - 9 T + 3 \lambda$ is irreducible over $\F_q$ if $q\geq 29$ is odd or $(1, \lambda, \lambda+1, 1)$, where $\lambda \in \F_q$ is such that $T^3+\lambda T^2 + (\lambda+1) T + 1$ is irreducible over $\F_q$ if $q\geq 32$ is even. The set $\cX = \cC \cup \{R\}$ is a complete NMDS-set. 
\end{prop}

\smallskip
{\footnotesize
\noindent\textit{Acknowledgments.}
This work was supported by the Italian National Group for Algebraic and Geometric Structures and their Applications (GNSAGA-- INdAM).
}


\begin{thebibliography}{SK}

\bibitem{AL} V. Abatangelo, B. Larato, Near-MDS codes arising from algebraic curves, {\em Discrete Math.}, 301 (2005), 5--19.

\bibitem{AGS} A. Aguglia, L. Giuzzi, A. Sonnino, Near-MDS codes from elliptic curves, {\em Des. Codes Cryptogr.}, 89 (2021), 965–972.

\bibitem{BS} B. Bagchi, N. Sastry, Intersection pattern of the classical ovoids in symplectic $3$-space of even order, {\em J. Algebra}, 126 (1989), 147--160.

\bibitem{BMP} D. Bartoli, S. Marcugini, F. Pambianco, On the completeness of plane cubic curves over finite fields, {\em Des. Codes Cryptogr.}, 83 (2017), 233--267.

\bibitem{magma} W. Bosma, J. Cannon, C. Playoust, The Magma algebra system. I. The user language, {\em J. Symbolic Comput.}, 24 (1997), 235--265.

\bibitem{DL} S. Dodunekov, I. Landgew, On near-MDS codes, {\em J. Geom.}, 54 (1995), 30--43.

\bibitem{GDT} E.M. Gabidulin, A.A. Davydov, L.M. Tombak, Linear codes with covering radius $2$ and other new covering codes, {\em IEEE Trans. Inform. Theory}, 37 (1991), 219--224.

\bibitem{G1} M. Giulietti, The geometry of covering codes: small complete caps and saturating sets in Galois spaces, {\em Surveys in combinatorics 2013, 51--90. London Math. Soc. Lecture Note Ser.}, 409, Cambridge Univ. Press, Cambridge, 2013. 

\bibitem{Giulietti} M. Giulietti, On the Extendibility of Near-MDS Elliptic Codes, {\em Appl. Algebra Eng. Commun. Comput.}, 15 (2004), 1--11.

\bibitem{H2} J.W.P. Hirschfeld, {\em Finite Projective Spaces of Three Dimensions}, Oxford Mathematical Monographs, Oxford Science Publications, The Clarendon Press, Oxford University Press, New York, 1985.

\bibitem{H1} J.W.P. Hirschfeld, {\em Projective Geometries over Finite Fields}, Oxford Mathematical Monographs, Oxford Science Publications, The Clarendon Press, Oxford University Press, New York, 1998.

\bibitem{HT} J.W.P. Hirschfeld, J.A. Thas,  {\em General Galois geometries}, Springer Monographs in Mathematics, Springer, London, 2016.

\bibitem{PW} T. Penttila, B. Williams, Ovoids of parabolic spaces, {\em Geom. Dedicata}, 82 (2000), 1--19.

\bibitem{WH} Q. Wang, Z. Heng, Near MDS codes from oval polynomials, {\em Discrete Math.}, 344 (2021), 112277, 10 pp.

\end{thebibliography}
\end{document}